\documentclass[12pt]{amsart}

\usepackage{amsthm,amsmath,mathrsfs}
\usepackage{amssymb,amsthm,amsmath}
\usepackage[dvips]{geometry}
\usepackage{amscd}
\usepackage[latin1]{inputenc}

\usepackage{colortbl}
\usepackage{color,graphics}
\usepackage{faktor}

\theoremstyle{plain}
\newtheorem{mainthm}{Theorem}

\newtheorem{mainclly}[mainthm]{Corollary}
\swapnumbers
\newtheorem{theorem}{Theorem}[section]
\newtheorem{example}{Example}
\newtheorem{proposition}[theorem]{Proposition}

\newtheorem{definition}[theorem]{Definition}

\newcommand{\T}{\mathbb{T}}

\newcommand{\al} {\alpha}       
        
\newcommand{\ga} {\gamma}    
\newcommand{\de} {\delta}

\newcommand{\Z}{\mathbb{Z}}
\newcommand{\N}{\mathbb{N}}
\newcommand{\R}{\mathbb{R}}
\newcommand{\eps}{\varepsilon}
\newcommand{\vphi}{\varphi}

\title{Expansive Lie Group Actions}

\author{Alexander Arbieto}
\address{Instituto de Matem\'atica, Universidade Federal do Rio de Janeiro, P. O. Box 68530, 21945-970 Rio de Janeiro, Brazil.}

\email{arbieto@im.ufrj.br}

\author{Elias Rego}

\address{Department of Mathematics, Southern University of Science and Technology, Shenzhen, Guangdong, China.}

\email{rego@sustech.edu.cn}

\thanks{Primary MSC code: 37B05, secundary MSC code:37B40. Keywords: Lie group actions, expansiveness,  entropy. 
 A. A. was partially supported by CNPq, FAPERJ and
PRONEX/DS from Brazil.}

\date{\today}

\begin{document}

\maketitle

\begin{abstract}
In this work we introduce a concept of expansiveness for actions of connected Lie groups. We study some of its properties and investigate some implications of expansiveness. We study the centralizer of expansive actions and introduce  $CW$-expansiveness for pseudo-group actions. As an application, we prove positiveness of geometric entropy for expansive foliations and expansive group actions. 
\end{abstract}

\section{Introduction}
 
 Expansiveness is a well established dynamical property. It raised from the smooth theory of dynamical systems, but very soon it was realized to be a feature of topological nature. W. R. Utz was the first to define expansiveness in the 50's in the setting of homeomorphims. Later, expansive systems were proved to be a huge source of rich dynamics. Indeed, they are closely related to stability phenomena,  chaos and entropy theory,  just to give a few examples. 
 
 Although expansiveness is a widely studied topic in discrete-time dynamics, the theory presents difficulties when one tries to extend it to more general contexts. For instance, if we try to move to continuous-time systems, we can face many obstructions. Indeed, the definition of expansiveness for homeomorphisms is not suitable to deal with flows, since there are not non-trivial flows satisfying it(\cite{BW}). This fact together with the necessity of dealing with reparametrizations made it appear many distinct definitions of expansiveness for flows. 
 
 The first one is due to R. Bowen and P. Walters in \cite{BW} (in this work we call it $BW$-expansiveness). Even though its great success to describe expansive phenomena for non-singular flows, $BW$-expansiveness has shown to be an inappropriate tool to deal with flows presenting singularities accumulated by regular orbits, such as the Lorenz Attractor. Later,  other definitions of expansiveness emerged to overcome these obstacles. We can mention $k^*$-expansiveness, separating flows, kinematic expansiveness, geometric expansiveness and others (See \cite{Ar} for details).
 
 Our goal in this paper is to extend the definition of expansiveness for  more general systems. Namely, we consider actions of more general groups on compact metric spaces and investigate how the expansive behaviour rules their dynamics. 
 
 There are some known efforts in order to study expansive group actions. In \cite{Hur},\cite{BDS} and \cite{RV}, the authors studied expansive actions of  finitely generated groups. In \cite{BRV} W. Bonomo, J. Rocha and  P. Varandas  introduced  expansive $\R^k$-actions and studied their centralizers. Here we consider expansiveness for actions of more general Lie groups. There are substantial distinctions between the flow and the Lie group actions scenarios. Maybe the simplest one is that $BW$-expansive flows on manifolds are non-singular and this implies the existence of regular orbit foliations, but for actions of more general groups this is not always the case. Indeed, in section 2 we see that even expansive actions do not have fixed points, they may not be locally-free. Previous fact implies that the orbits of an expansive action may not generate a regular foliation of the phase space. To avoid singular orbit foliations, in this work we only consider locally-free actions. We point out that the precise definitions of the concepts appearing in this introduction  are postponed until section 2.

A question we address in this work is what scenarios allow the existence of expansive actions. We start recalling that if we are considering flows, then it is proved in  \cite{HS} that expansiveness cannot occur in closed surfaces. We investigate this result for general groups. In the general setting, the surface flows case is  translated into the codimension one actions case. By a codimension one action, we mean a $G$-action, $\varPhi$, on $M$ satisfying $codim(\varPhi)=\dim(M)-\dim(G)=1$. In this scenario, we obtain the following result.

\begin{mainthm}\label{EC}
   Let $M$ be a closed Riemannian manifold and let $\varPhi$ be a locally-free $C^2$-action of a nilpotent and connected Lie group $G$ on $M$. If $codim(\varPhi)=1$, then $\varPhi$ cannot be expansive. 
\end{mainthm}

   A celebrated  way to measure the complexity of a dynamical system is its topological entropy. It is well known that for homeomorphism and flows the expansiveness property is a source of complex dynamical behavior. For instance,  we cite the works \cite{Fa} and \cite{Ka}, where it is  proved that any expansive homeomorphism must have positive topological entropy, if the phase is space is rich enough. A similar result was proved for non-singular expansive flows in \cite{ACP}. Here we address the problem of the positiveness of entropy for expansive group actions, but our treatment considers the geometric entropy introduce by E. Ghys, R. Langevin and P. Walczak in \cite{GLW}, instead of topological entropy. This is because the geometric entropy is a version of entropy suitable  to deal with actions. In order to obtain our results, we chose to start working in a more general setting. We begin working with actions of pseudo-groups of local-homeomorphisms on compact metric spaces which are a generalization of continuous group actions on compact metric spaces. In addition, we introduce a weaker form of expansiveness for such actions, the so called $CW$-expansiveness (see section 3, for precise definition). We remark that expansive pseudo-groups are $CW$-expansive. In this general setting, we obtain the following result:

\begin{mainthm}\label{EntPG}
	Let $M$ be a compact, infinite and locally connected metric space. Suppose that $\mathcal{G}$ is a finitely generated pseudo-group of local homeomorphims of $M$. If $\mathcal{G}$ is $CW$-expansive, then it has positive geometric entropy. 
 \end{mainthm}
 
 Now suppose that  $M$ is a closed smooth manifold. Recall that if $\mathcal{F}$ is a regular foliation of $M$, then it has a natural associated pseudo-group of local homeomorphims, the so called holonomy pseudo-group of $\mathcal{F}$. In \cite{IT} T. Inaba and N. Tsuchiya introduced a concept of expansiveness for foliations. Precisely, a foliation is expansive, if its holonomy pseudo-group is expansive. In addition they proved that any expansive codimension-one foliation has positive geometric entropy. In higher condimension, they obtained positiveness of entropy under stronger expansive assumptions. We say that a foliation is  $CW$-expansive if its holonomy pseudo-group is $CW$-expansive. In particular, the foliations introduced in \cite{IT} are $CW$-expansive. As a consequence of Theorem \ref{EntPG} we obtain the following:

 \begin{mainclly}\label{EntFol}
If $\mathcal{F}$ is a $CW-$expansive $C^1$-foliation with positive codimension, then $\mathcal{F}$ has positive geometric entropy.
\end{mainclly}
 
We remark that the above result implies the results in \cite{IT} in any positive codimension and without stronger expansiveness assumptions.
 
Let us now come back to the group actions scenario. Suppose that $G$ is a connected Lie Group and $\varPhi$ is a $C^r$-action of $G$ on $M$. It is well known that  if the action is locally-free, then the set of orbits of $\varPhi$ generates a regular foliation of $M$. In this setting we can apply Corollary \ref{EntFol} to group actions and obtain entropy results.   
 
\begin{mainclly}\label{ExpEnt}
Let $\varPhi$ be a locally-free $C^1$-action of a connected Lie group  on a closed manifold such that $codim(\varPhi)>0$. If $\varPhi$ is expansive, then $\varPhi$ has positive geometric entropy.
\end{mainclly}

If $G=\mathbb{R}$, then $\varPhi$ is in fact a flow. In that case, our definition of expansiveness, implies $BW$-expansiveness. On the other hand, $BW$-expansive $C^r$-flows on closed manifolds, must be regular. This implies that their orbits generates an expansive foliation of the phase space. Recall that in this case, we have that the geometric entropy of the orbit foliation of $\varPhi$ is twice the topological entropy of the flow $\varPhi$, if we reparametrize $\varPhi$ with unitary speed. In addtion, the positiveness of the  topological entropy of a regular flow is invariant under  reparametrizations, then our results also have as consequence the result in \cite{ACP}.

We also studied the symmetries of expansive actions. There are many efforts in the direction of understanding the symmetries of $C^r$-actions. We reefer the reader to the works of D. Obata, M. Leguil and B. Santiago for actions of $\Z$ and $\R$ (\cite{LOS} and \cite{O}). In the setting of expansive homeomorphisms P. Walters proved in \cite{W} that such systems have discrete centralizer.  In \cite{BRV} the authors proved that expansive $\R^k$-actions have quasi-trivial centralizers. Here we  extend these results,
but in order to do it, we need an additional hypothesis on $G$. Recall that if $G$ is a Lie group, then its Lie algebra $\mathfrak{G}$ is isomorphic to $T_eG$, where $e$ is the identity element of $G$. So we can see $\mathfrak{G}$ as an  additive group.

\begin{mainthm}\label{CC}
Let $\varPhi$ be  an expansive $C^1$-action of a group $G$ on a closed manifold and let $\mathfrak{G}$ be the Lie algebra of $G$. If the exponential map $\exp:\mathfrak{G}\to G$  is a group-homomorphism, then $\varPhi$ has quasi-trivial centralizer. 
\end{mainthm}

\vspace{0.1in}
\textit{The Finitely Generated Case}
\vspace{0.1in}

We also investigate some consequences of expansiveness in the case where $G$ is a finitely generated group. Recall that a continuous group action on $M$  can be seen as a pseudo-group action of $M$. With this interpretation, the following corollary is  another consequence of Theorem \ref{EntPG}.

\begin{mainclly}\label{FG}
	Let $M$ be a infinite, compact and locally connected metric space. Suppose that $\varPhi$ is a  continuous action of a finitely generated  group $G$ of homeomormphisms of $M$. If $\varPhi$ is $CW$-expansive, then it has positive geometric entropy. 
 \end{mainclly}

This result generalizes the results in \cite{Hur} for expansive actions of the circle. In addition, if $G=\mathbb{Z}$, then $\varPhi$ is the action of a single homeomorphism on $M$. In this case we have that the geometric entropy of $f$ is twice its topological entropy. As a consequence, Corollary \ref{FG} implies the classical results in \cite{Ka} and \cite{Fa}. 
Finally, we address the question of the centralizer of expansive actions in the finitely generated case. In this setting, we obtained the following generalization of the results in \cite{B}.

\begin{mainthm}\label{DC}
The centralizer of any expansive $C^0$-action of a finitely generated group G on a closed manifold is a discrete set on the space of $C^0$-actions of $G$ on $M$.

\end{mainthm}

This paper is divided as follows:
\begin{itemize}
    \item In section 2 we define the basic concepts needed to study group actions. We define expansiveness and explain the difference between the definitions for finitely generated and connected groups. We  also study some of the basic properties of expansive actions of connected Lie groups and prove Theorem \ref{EC}

\item In section 3 we  study the entropy problem for expansive actions. We introduce the concept of $CW$-expansiveness for pseudo-groups and prove Theorem \ref{EntPG}, Corollary \ref{EntFol} and Corollary \ref{ExpEnt}.

\item In section 4 we study the centralizer of expansive actions of connected groups and prove Theorem \ref{CC}.

\item In section 5 we  study in more detail the finitely generated case. We give some examples and prove Theorem \ref{DC}

\end{itemize}
\section{Preliminaires}

In this section we define the basic concepts needed to the study of group actions and the concept of expansiveness.

\vspace{0.1in}
\textit{Group Actions.}
\vspace{0.1in}

 Throughout this work  $M$ denotes a compact metric space with a metric $d$. Some times we assume that $M$ is a Riemannian manifold, in this case $d$ denotes the metric induced by the Riemannian metric of $M$. Let us denote $\mathcal{K}(M)$ 
for the compact  hyperspace of $M$, i.e., 
 $$\mathcal{K}(M)=\{K\subset M; K\textrm{ is compact}\} $$
We also denotes $\mathcal{C}(M)$ for the continuuum hyperspace of $M$, i.e.,
$$\mathcal{C}(M)=\{C\subset M; C\textrm{ is  compact and connected}\}.$$
Recall that both $\mathcal{K}(M)$ and $\mathcal{C}(M)$ can be made into compact metric spaces through the Hausdorff Metric
$$d^H(A,B)=\max\{\sup\limits_{x\in A}d(x,B),\sup_{x\in B}d(x,A)  \}.$$

Let $G$ be a Lie group of dimension $k$ and  endow $G$ with a right invariant metric  $d^G$. We save $|g-h|$ to denote $d^G(g,h)$ and  $|g|$ to denote the distance from $g$ to the identity element $e$ of $G$.  We  denote  $End(G)$ for the set of group endomorphisms of $G$, i.e., the set of group homomorphisms $h:G\to G$. Recall that to any Lie group $G$ we can associate a Lie algebra $\mathfrak{G}$ which is isomorphic to $T_eM$.  

\begin{definition}
 An action of $G$ on $M$ ($G$-action on $M$) is a map $\varPhi: G\times M\to M$ satisfying:
\begin{itemize}
\item $\varPhi(e,\cdot)=Id_M$  
\item $\varPhi(hg,x)=\varPhi(h,\varPhi(g,x))$.
\end{itemize}
We say that a $G$-action on $M$ is continuous if the map $\varPhi$ is continuous. If $M$ is a smooth manifold, $r\geq 1$ and the map $\varPhi$ has regularity $C^r$, we say that $\varPhi$ is a $C^r$-action. Let us denote $\mathcal{A}^0(G,M)$ for the set of continuous $G$-actions on $M$ and $\mathcal{A}^r(G,M)$ for the set of $C^r$-actions of $G$ on $M$. 

\end{definition}

 As usual, we use the notation $gx:=\varPhi(g,x)$ and $\phi_g(\cdot)=\varPhi(g,\cdot)$. In addition,  $B_{r}(x)$, $B_{r}^G(x)$ and $B_r^{\mathfrak{G}}(x)$ will stand for the open balls of $M$, $G$ and $\mathfrak{G}$, respectively.

\vspace{0.1in}
\textit{Foliations.}
\vspace{0.1in}

 Next we recall the concept of foliations for smooth manifolds and some of their properties. We refer the reader to \cite{Wa} for a complete exposition of the concepts and results presented here.
\begin{definition}
   Let $M$ be a smooth manifold. A p-dimensional $C^r$ foliation $\mathcal{F}$ is a decomposition of $M$ in to connected $C^r$-manifolds (the leaves of $\mathcal{F}$) such that for any $x\in M$, there is a $C^r$-differentiable chart $\xi=(\xi',\xi''):U\to \R^d=\R^p\times \R^q$ defined on a neighborhood of $U$ and satisfying:
   
   (*) - For any leaf $L$, each connected component of   $L\cap U$ (called plaque) satisfies $\xi''=c$, where $c$ is constant
   
\end{definition}

Charts satisfying (*) are called distinguished charts. An atlas constituted by distinguished charts is called a foliated atlas. The number $q$ in the previous definition is the codimension of $\mathcal{F}$. The topology of the leaves as manifolds is in general stronger than that induced by the topology of $M$. We say that a leaf $L$ is proper if these two topologies coincide. 

\textit{Remark:} Compact leafs are always proper, but the converse does not hold.

\begin{definition}
   A foliated atlas $\mathcal{A}=$ is nice if it satisfies:
   
   \begin{itemize}
       \item The covering $\{U_{\xi}\}_{\xi \in \mathcal{A}}$ is locally finite.
       
       \item For any $\xi\in \mathcal{A}$, we have that $\xi(U_{\xi})$ is an open cube of $\R^d$
       \item If $\xi_1,\xi_2\in \mathcal{A}$ and $U_{\xi_1}\cap U_{\xi_2}\neq \emptyset$, then there exists a distinguished chart $\xi_3\in \mathcal{A}$  such that $U_{\xi_3}$ contains the closure of $U_{\xi_1}\cup U_{\xi_2}$ and $\xi_3|_{U_{\xi_1}}=\xi_1$.
   \end{itemize}
\end{definition}

\begin{theorem}
There always exist nice atlases for a foliation $\mathcal{F}$. 
\end{theorem}

Let $\mathcal{F}$ be a foliation of $M$. If $\mathcal{A}$ is a nice atlas for $\mathcal{F}$,  We call $\mathcal{U}=\{U_{\xi}\}_{\xi\in\mathcal{A}}$ a nice covering for $M$. For every $U\in \mathcal{U}$, define in $U$  the following equivalence relation:
$$x\sim y \textrm{ if, and only if } x \textrm{ and } y \textrm{ are in the same plaque of } U.$$
In this way Then $T_U=\faktor{U}{\sim}$ can be equipped with the quotient topology.  $T_U$ is called the set of the plaques of $U$ and $T_U$ is $C^r$-diffeomorphic to an open cube of $\R^q$, where $q$ is the codimention of $\mathcal{F}$ and $r$ is the regularity of $\mathcal{F}$. Furthermore, $T_U$ can be immersed in $M$ as a manifold transverse to every leaf of $U$. By this reason, we will often call $T_U$ a local transversal to $\mathcal{F}$ and we will always see it as an immersed manifold in $M$     

\begin{definition}
   The disjoint union $$\mathcal{T}=\bigcup_{U\in \mathcal{U}}T_U$$
is a complete transversal for $\mathcal{F}$.
\end{definition}

\vspace{0.1in}
\textit{Expansiveness.}
\vspace{0.1in}

Now we  discuss some aspects of expansiveness in distinct contexts. As we will see, expansiveness must be defined in different ways depending on the kind of system in consideration.  The notion of expansiveness was introduced by Utz in his seminal work \cite{U}. At that time, expansive systems were called unstable systems and later the current denomination was established. Essentially, expansiveness is a property which separates any pair of  distinct points  by a uniform constant in some instant of time.

\vspace{0.1in}

\begin{definition}
 A homeomorphism $f:M\to M$ is expansive if there exists a constant $e\delta>0$ such that for any two distinct points $x,y\in M$, we can find $n\in \Z$ such that $d(f^n(x),f^n(y))>\delta$.    

\end{definition}

This concept was defined for flows by R. Bowen and  P. Walters in \cite{BW}, but with some distinctions. Recall that a flow is an $\R$-action on $M$.

\vspace{0,1in}

\begin{definition}
    
We say that a continuous flow $\phi$ on $M$ is  BW-expansive if for any $\eps>0$ there is some  $\delta>0$ such that if there exist $x,y\in M$ and a continuous function $h:\R\to \R$  such that $h(0)=0$ and  $d(\phi_t(x),\phi_{h(t)}(y))<\delta$ for every $t\in \R$, then there exists $t_0\in \R$ with $|t_0|<\eps$ such that $y=\phi_{t_0}(x)$.

\end{definition}

Let us explain the distinction between the above definitions. First, the connectedness of $\R$  together with the group property of $\phi$ implies the non-existence of a non-trivial flow satisfying a definition of expansiveness  similar to  the one for homeomorphisms (see \cite{BW}). The second reason is the necessity to deal with time-changes in continuous-time systems. In view of the above distinctions, if one wants to define expansiveness for more general group actions, then one must be careful about the kind of structure the acting group has. Because of this,  we shall divide our study in two cases, namely, finitely generated groups and connected Lie groups.

For the finitely generated case, a concept of expansiveness is already well established (see  \cite{Hur} , \cite{BDS} and \cite{RV}). It is a  transportation of Utz's definition to the context of group actions. Precisely, we have the following:
\vspace{0.1in}

\begin{definition}\label{expdiscdef}Let $G$ be a finitely generated group. A  continuous $G$-action on $M$ is said to be expansive if there exists a constant $\delta>0$ such that for any two distinct points $x,y\in M$, we can find $g\in G$ such that $d(gx,gy)>\delta$.

\end{definition}    

 Later, a definition for $\R^d$-actions was given in \cite{BRV}. Their definition is a natural extension of $BW$-expansive flows.  Next we present our definition of expansiveness for actions of connected Lie groups,  extending the definition in \cite{BRV}.  

\begin{definition}\label{expcontdef}
Let $G$ be a connected Lie group. A continuous $G$-action on $M$ is said to be expansive if for any $\eps>0$ there exists $\delta>0$ such that the following holds: 

If there exist $x,y\in M$ and a continuous map $h:G\to G$  fixing the identity  element of $G$ and verifying $d(gx,h(g)y)<\delta$ for every $g\in G$, then there exists $g_0\in G$ with $|g_0|<\eps$ such that $y=g_0x$.
\end{definition}

We notice that our definition is more general to the one presented in \cite{BRV}, since that definition is the same as ours, but restricting $G$ to be $\R^d$. Our first result states that expansiveness is a dynamical property. Recall that two $G$-actions $\varPhi$ and $\varPsi$ on  $M$ and $N$, respectively,  are conjugated if there exists a homeomorphism $f:N\to M$ satisfying the following conjugacy equation, for any $g\in G$: $$\varPhi_g\circ f=f\circ \varPsi_g$$   

\begin{theorem}
	Any action conjugated to an expansive action is expansive.
\end{theorem}
\begin{proof}
	Suppose that $\varPhi$ is an expansive $G$-action on $M$ which is conjugated to a $G$-action, $\varPsi$, on $N$ and fix $\eps>0$. Let $\de_1>0$ be given by the expansivity of $\varPhi$. Let $f:N\to M$ be the conjugacy homeomorphism. Then one can choose $\de>0$  such that $d(f(x),f(y))<\de_1$ if $d(x,y)<\de$ for every $x,y\in N$. We claim that $\de$ is an expansivity constant of $\varPsi$. Indeed, if there are $x,y\in N$ and  continuous map $h:G\to G$ fixing $e$ and verifying   $d(\varPsi_g(x),\varPsi_{h(g)}(y))<\de$ for every $g\in G$, then $d(\varPhi_g(f(x)),\varPhi_{h(g)}(f(y)))<\de_1$ for every $g$. Then $f(y)\in\varPsi(B_{\eps}(e),f(x))$ and therefore $y\in\varPsi(B_{\eps}(e),x)$.  This completes the proof.  
\end{proof}

The following proposition  implies the non-existence of fixed points for expansive actions on connected metric spaces and its proof is exactly the same as in the case of flows. 

\begin{proposition}\label{PFF}
 If  $\Phi$ is expansive, then any fixed point of $\Phi$ is a isolated point of $M$.
\end{proposition}

\begin{proof}

Suppose that $p$ is fixed for $\Phi$. Take $\eps>0$ and let $\de>0$ be given by the expansiveness. If we take the constant map $h\equiv e$, then $d(gp,h(g)y)=d(p,y)<\de$ for any $g\in G$ and $y\in B_{\de}(p)$. Expansiveness implies that $y=p$ and then $B_{\de}(p)=\{p\}$ and $p$ is isolated.
\end{proof}

Previous proposition implies that the Definition \ref{expcontdef} is, in fact, distinct from Definition\ref{expdiscdef}. Indeed, there are expansive homeomorphisms in sense of Definition \ref{expdiscdef} with fixed points on connected spaces, but this forbids these systems to satisfy Definition \ref{expcontdef}. This endorses the necessity of dividing the study of expansive  actions  in the finitely generated and connected cases.

Hereafter, we treat the case of connected Lie group actions and postpone the discussion about the finitely generated case until the last section.
 Henceforth we are under the following assumptions, unless otherwise stated:
 
 \begin{itemize} 
 \item $G$  denotes a connected Lie group. 
 \item $M$  denotes a compact, connected and boundaryless  Riemannian manifold
 \item All the  $G$-actions considered here are of class $C^1$.
 \end{itemize}
 
 Our first  result states that compact groups cannot act expansively.

\begin{theorem}
Let $G$ be a compact group. If $\varPhi$ is continuous and expansive, then $M$ is a single orbit.
\end{theorem}

\begin{proof}
Suppose $G$ is compact and $M$ is not a single orbit. Since $G$ is compact, then  $diam(G)$ is finite and therefore we can find $R>0$ such that $B^G_R(e)=G$. If we fix $\eps>0$, by continuity we can find some $\de>0$ such that if $d(x,y)\leq \de$, then $d(gx,gy)<\eps$, for any $|g|\leq R$. Since $M$ is not a single orbit, we can find $x\in M$ and $y\in M\setminus(O(x))$ satisfying $d(x,y)\leq \de$. But since $G=B_R(e)$, this an obstruction for expansiveness because we have that $d(gx,gy)\leq \eps$ for any $g\in G$.

\end{proof}

 Actions of connected Lie groups are closely related with foliations theory. Note that Proposition \ref{PFF} implies that the set of orbits of any expansive $C^r$-flow forms  a natural $C^r$-foliation $\mathcal{F}_{\varPhi}$ of $M$  satisfying $\dim(\mathcal{F}_{\varPhi})=\dim(G)$. For a higher dimensional group $G$, the previous condition is equivalent to the action to  be locally free, i.e., for any $x\in M$, its isotropy subgroup $G_x\subset G$ is discrete. Here the isotropy subgroup of $x$ is defined to be the set $$G_x=\{g\in G; gx=x\}. $$  
It is a classical fact that an action is locally-free if, and only if, its orbits form a foliation with the same dimension as $G$.

\begin{example}One can obtain an example of a locally-free $\R^2$-action through the suspension process described in \cite{BRV}. Indeed, if one considers the diffeomorphisms $f_A$ and $f_B$ on $\mathbb{T}^3$ induced by the  hyperbolic matrices $$A=\begin{pmatrix}3&2&1\\2&2&1\\1&1&1\end{pmatrix} \textrm{ and } B=\begin{pmatrix}2&1&1\\1&2&0\\1&0&1\end{pmatrix},$$   it is easy to see that $f_A$ and $f_B$ commute. Moreover, $f_A$ and $f_B$ are expansive, since they are anosov diffeormorpims. $f_A$ and $f_B$ generates an expansive  $\mathbb{Z}^2$-action on $\mathbb{T}^3$ that can be suspended to a locally-free $\mathbb{R}^2$-action on a manifold $M$ with dimension $5$. According the results in \cite{BRV}, this action must be expansive.  
\end{example}

On the other hand, for groups distinct from $\mathbb{R}$,  the non-existence of fixed points is not enough to guarantee  that the action is locally free. Next example illustrates this fact.

\begin{example}
Let $M$ be a smooth closed manifold and $X$ be a $C^r$-vector field generating a $BW$-expansive flow with non-trivial centralizer. Let $Y$ be a non-trivial vector filed commuting with $X$. By Corollary 2 of \cite{BRV} the centralizer of $M$ is quasi-trivial.  Then $Y$ generates a flow with the same orbits of $X$. Then the $\R^2$-action generated by $X$ and $Y$ is an expansive action with orbits of dimension one and then it cannot be locally free.
\end{example}

Since non-locally free actions may not generate non-singular foliations,  their study is much more challenging. Because of this, in this work we will always suppose that $\varPhi$ is a locally free group action. The next proposition will be useful through next sections.

\begin{proposition}\label{iso}
If $\varPhi$ is a locally-free action of $G$ on $M$, then there exists $\de>0$ such that $G_{x}\cap B^G_{\de}(e)=\{e\}$ for any $x\in M$.  
\end{proposition}

\begin{proof}
This is a trivial consequence of the fact that $\varPhi$ is a foliated action.
Indeed, if the result is false, we can obtain a sequence of points $x_n\in M$ such that $$G_{x_n}\cap B^G_{\frac{1}{n}}(e)\neq\{e\}$$
Assume that $x_n\to x$. Since $\varPhi$ is foliated, then there is some $\eta>0$ and a local transversal $T_x$ to $x$ such that $U_{\eta,x}=\varPhi(B_{\eta}(e),T_x)$ is a foliated neighborhood of $x$ and $\varPhi(B_{\eta}^G(e),y)\cap T_x=\{y\}$ for any $y\in T_x$. Using the continuity of the action, we can find $n$ big enough such that: \begin{itemize}
    \item $\frac{1}{n}\leq \frac{\eta}{2}.$ 
    \item $\varPhi(B^G_{\frac{1}{n}}(e),x_n)\subset U_{\eta,x}$ 
    \item $x_n=g_nx_n$ for some $g_n\neq e$ with $|g_n|\leq \frac{1}{n}$
    \item $\varPhi(\ga_n,x_n)\subset U_{\eta,x}$, where $\ga_n$ is the geodesic on $G$ connecting $e$ and $g_n$.
    
    \end{itemize}
    This forces the orbit of $x_n$ to intersect $T_x$ twice in $U_{\eta,x}$ for small time, but this is impossible by the choice of $\eta$. Thus the proposition is proved.  
\end{proof}

Since locally-free actions are foliated, we can use the porperties of its orbit foliation to analyse the behaviour of actions. Next we will define a concept of expansiveness for foliations introduced by T. Inaba and M. Tsuchiya in \cite{IT}. Let  $\mathcal{F}$ be a folitation of $M$. By the compactness of $M$ there exists $\eps_0>0$ such that for any $0<\eps\leq\eps_0$ the  local transversal disc  $D_{\eps}(x)$ at $x$ with radius $\delta$ is well defined. For such an $\eps>0$, we always can consider a complete transversal $\mathcal{T}$ whose its elements are discs of the form $D_{\eps}(x)$. A $\mathcal{F}$-curve is a curve contained in some leaf of $\mathcal{F}$. Fix  some $\mathcal{F}$-curve $\al:[0,1]\to M$ and  let $N$ be some local transversal disc  containing $\al(0)$.

A fence $F$ along $\alpha$ is a continuous map $F:[0,1]\times N\to M$ such that:
\begin{itemize}
    \item $F|_{\{t\}\times N}$ is an embedding of $N$ into a transversal disc  $D_{\eps}(\al(t))$ for any $t$.
    \item $F|_{[0,1]\times\{x\} }$ is a $\mathcal{F}$-curve for any $x\in N$
    \item There exists $x_0\in N$ such that $F|{[0,1]\times\{x_0\} }=\al$.
\end{itemize}

\begin{definition}
	$\mathcal{F}$ is said to be expansive if there exists $\de>0$ such that for any $x\in M$ and $y\in D_{\de}(x)\setminus \{x\}$ there exists a $\mathcal{F}$-curve $\al$  such that $\al(0)=x$ and a fence $F$ along $\al$ such that $F(1,y)\notin D_{\de}(\al(1))$. 
	
\end{definition}

\begin{theorem}\label{ExpFol}
	The orbit foliation of a locally-free expansive action is expansive.
\end{theorem}

\begin{proof}
	Suppose $\Phi$ is an expansive locally free action and fix some complete transversal $\mathcal{T}$ to the orbit folialion of $\Phi$. Proposition \ref{iso} allows us to chose $\eps>0$ such that for any $x\in M$ one has $\vphi(B_{\eps}(e),x)\cap T(x)=\{x\}$. Let $\de>0$ be the expansive constant related to $\eps$. Now fix $x\in M$ and  take $y\in D_{\de}(x)$. By the choice of $\eps$, we have that $y\notin\vphi(B_{\eps}(e),x)$. Then expansveness gives us a $g_0\in G$ such that $d(g_0x,g_0y)>\de$. Now let $\gamma:[0,1]\to G$ be a curve connecting $e$ and $g_0$.  Define the map $F([0,1]\times D_{\de}(x))\to M$ by $$F(t,p)=\vphi(\gamma(t),p)$$.
	
	Previous map is clearly a fence satisfying $F(1,y)\notin D_{\de}(g_0x)$ and therefore the orbit foliations of $\vphi$ is expansive.
	
\end{proof}

Previous result allows us to prove Theorem \ref{EC}.

\begin{proof}[Proof of Theorem \ref{EC}]
Let $\varPhi$ be a locally-free and codimension-one group action of a nilpotent Lie group $G$ on a closed manifold $M$. In \cite{HGM}, Hector, Ghys and Moriyama showed that the orbit foliation of $\varPhi$ is almost without holonomy, i.e., every non-compact leaf has trivial holonomy. On the other hand, if $\varPhi$ is expansive, then Theorem \ref{ExpFol} combined with the Corollary 2.6 in \cite{IT} implies that some orbit of $\varPhi$ is a resilient leaf. But a resilient leaf a non-compact and has non-trivial holonomy.
This is a contradiction, and then Theorem A is proved.

\end{proof}

\section{Geometric Entropy of Expansive Pseudo-Group Actions}

Now we begin to investigate the relationship between the expansiveness and the geometric entropy for connected Lie Group Actions. Let us begin recalling the concept of a pseudo-group of local homeomorphisms.

\begin{definition}
   Let $X$ be a topological space. A peseudo-group $\mathcal{G}$ on $X$ is a family $$\mathcal{G}=\{h:D_h\to R_h\}$$ of local homeomorphism of $X$ satisfying the followinng conditions:

\begin{enumerate}
    \item If $h,g\in \mathcal{G}$ and $D_g\subset R_h$, then $g\circ h\in \mathcal{G}$
    \item If $g\in \mathcal{G}$, then $g^{-1}\in \mathcal{G}$
    \item If $g\in \mathcal{G}$ and $U\subset D_g$, then $g|_U\in \mathcal{G}$
    \item If $g$ is a local homeomorphism of $X$ and $\mathcal{U}$ is an open cover of $D_g$ such that $g|_{U}\in \mathcal{G}$, then $g\in \mathcal{G}$.
    \item $Id_X\in \mathcal{G}$.
\end{enumerate}
\end{definition}

For any set $\Gamma$ of local homeomorphisms of $X$ satisfying $$\bigcup_{g\in \Gamma} (D_g \cup R_g)=X,$$ 
we can always find a pseudo-group $\mathcal{G}(\Gamma)$ of local homeomorphisms of $X$ generated by $\Gamma$, i.e, for any $g\in \mathcal{G}(\Gamma)$ and $x\in D_g$ one can find $g_1,...,g_i\in \Gamma$, $e_1,...,e_i\in \{-1,1\}$ and a neighborhood $U$ of $x$ satisfying $$g|_U=g_i^{e_i}\circ\cdots\circ g_1^{e_1}|_U.$$

\begin{definition}
   We say that a pseudo-group $\mathcal{G}$ is finitely generated if there is a finite set $\Gamma\subset \mathcal{G}$ such that $\mathcal{G}=\mathcal{G}(\Gamma)$.

\end{definition}

Remember that if $\mathcal{F}$ is a foliation of $M$, the it has a natural associated pseudo-group $\mathcal{G}$, namely, its holonomy pseudo-group. When $M$ is compact, we have that $\mathcal{G}$ is finitely generated (see \cite{Wa}). Next we will describe a natural way to obtain a finite generator for the holonomy pseudo-group of the orbit foliation of a locally-free action. 

Let $\varPhi$ be a locally-free action on $M$ and fix some point $x\in M$. Let $\mathcal{T}$ be a complete tranversal to the orbit foliation of $\varPhi$. Since the action is locally free, for every $\eps>0$ we can find $\de_x>0$ such that $T_x=B_{\de_x}(x)\cap T(x)$ is a local cross-section of time $\eps$ for the action $\varPhi$ through $x$, i.e., $\varPhi(B_{\eps}(e),y)\cap T_x=\{y\}$, for every If $y\in T_x$. By compactness of $M$ we can find $\{x_1,...,x_n\}\in M$ such that $$\bigcup_{i=1}^n\varPhi(B_{\eps},T_{x_i})=M$$ 

The last condition, implies that the holonomy maps between the cross-section $T_{x_i}$ generates the holonomy pseudo-group of the orbit foliation of $\varPhi$. Note that these conditions are totally analogous to the techniques of cross-sections developed by R. Bowen and P. Walters in \cite{BW} and by H.B. Keynes and M. Sears in \cite{KS} to study $BW$-expansive flows. 

In \cite {IT}, the authors proved that any expansive codimension one foliation has positive entropy. This is a consequence of the existence of resilient leaves. For higher codimensional expansive foliations, they proved the same result under stronger assumptions on the expansiveness of $\mathcal{F}$. Our main goal on this section is to prove the positiveness of geometric entropy for expansive pseudo-group actions. In particular,  this weakens the stronger assumption in \cite {IT} and implies positive entropy  for any positive dimensional and expansive  foliation. In particular, it will imply that expansive actions of connected Lie groups do have positive geometric entropy. First let us recall, the definition of geometrical entropy for pseudo-groups. 

Let $\mathcal{G}$ be a finitely generated pseudo-group of local homeomorphims of a compact metric space $M$, with a generator $G$. Let $g\in \mathcal{G}$. We say that $g$ has size $k$ and denote $\#g=k$ if the minimum amount of elements of $G$ needed to write $g$ is $k$. Precisely, we can write $g=g_{i_1}\circ...\circ g_{i_k}$, with $g_{i_1},...,g_{i_k}\in G$ and it is not possible to write $g$ with less than $k$ elements of $g$. Fix some $\eps>0$ and some natural $n$.  We say that a pair of points $x,y\in M$ is $n$-$\eps$-$G$-separated by $\mathcal{G}$ if  there exists $g\in \mathcal{G}$ such that $x,y\in D_g$, $\#g\leq n$ and  $d(g(x),g(y))>\eps$. 

A subset $E\subset M$ is $n$-$\eps$-$G$-separated if any pair of its distinct points is $n$-$\eps$-$G$-separated by $\mathcal{G}$. Let $S(n,\eps,G)$ denote the maximal cardinality of a $n$-$\eps$-$G$-separated subset of $M$.

\begin{definition}

The topological entropy of $\mathcal{G}$ with respect to $G$ is defined to be 

$$h(\mathcal{G},G)\lim_{\eps\to 0}\limsup_{n\to \infty}\frac{1}{n}\log S(n,\eps,G)     $$

\end{definition}

It is a classical fact that if a pseudo-group has positive entropy with respect to some finite generator $G$, then it also has positive entropy with respect to any other finite generator(see \cite{Wa}). Next concept is a weaker notion of expansiveness for actions of pseudo-groups.

\begin{definition}
A pseudo-group $\mathcal{G}$ is $CW$-expansive if there is some $\delta>0$ such that if $C\subset M$ is non-trival compact and connected set,
 there is some $g\in \mathcal{G}$ such that $C\subset D_g$ and $diam(g(C))>\delta$. The constant $\delta$ is called the expansiveness constant of $\mathcal{G}$.
\end{definition}

Since groups of homeomorphisms are also pseudo-group, then a trivial example of $CW$-expansive pseudo group is a group of homeomorphims containing a $CW$-expansive homeomormphim (see \cite{Ka}). A non-trivial example is the holonomy pseudo-groups of the orbit foliation of a $CW$-expansive flows (see \cite{ACP}). It is immediate that the concept of expansiveness for pseudo-groups presented in \cite{Wa} implies $CW$-expansiveness. In addition, the concept of expansiveness introduced in \cite{Wa} is equivalent to expansiveness for foliations.  Then by Theorem \ref{ExpEnt} the holonomy pseudo-group of the orbit foliation of an expansive action is $CW$-expansive. Next theorem states that $CW$-expansive and finitely generated pseudo-groups satisfy a uniform version of expansiveness.

\begin{theorem}\label{unexpo}
    Let $\mathcal{G}$ be a finitely generated pseudo-group  of local homeomorphisms of a compact metric space $M$. If $\mathcal{G}$ is $CW$-expansive with expansive constant $\delta$, then there are a finite generator $G$ and $\eps_0>0$   such that for any $0<\eps<\eps_0$ , there is some $N\in \N$ such that:
    
    If $C\subset M$  satisfy $\eps\leq diam(C)< \eps_0$, then there is $g\in \mathcal{G}$ with $\#g\leq N$, such that $C\in D_g$ and $diam(g(C))>\delta$.
\end{theorem}

\begin{proof}
    We start fixing a finite generator $G$ of $\mathcal{G} $ such that $D=\{D_g\}_{g\in G}$ is an open cover for $M$. Fix $\eps_0>0$ the Lebesgue number of the cover $D$.  
    Suppose that $\mathcal{G}$ is expansive with constant $\eps_0$ and suppose that we can find $0<\eps<\eps_0$ with the following property: For every $n\in \N$, there are $g_n\in G$ and compact and connected sets $C_n\subset D_{g_n}$ such that $diam(C_n)\geq \eps$ and $diam(g(C_n))\leq \delta$, for every $g$ such that $C_n\subset D_{g}$ and $\#g\leq n$.
 
 Since $G$ is finite we can assume that $D_{g_n}=D_g$ for every $n\in \N$ and some $g\in G$. Now the compactness of the continuum hyperspace allows us to assume that $C_n\to C$ in the hausdorff topology. Moreover, we have $C\subset D_g$ and  $diam(C)\geq \eps$.  Since $G$ is finite, then the set $$G_k=\{g;\#g\leq k\}$$ is also finite. If we fix any $g\in G_k$, then we have that if $n>k$, then  $diam(g(C))\leq \de$. Since $g$ is fixed, by continuity we have that $$diam(g(C))\leq\lim\limits_{n\to \infty}diam(g(C_n))\leq\de$$
 
 Since $k$ and $g$ were chosen arbitrarily,  this contradicts the $CW$-expansiveness of $\mathcal{G}$.

\end{proof}

Now we are able to prove Theorem \ref{EntPG}.

\begin{proof}[Proof of Theorem \ref{EntPG}]

 Let us fix $\eps_0>0$ and a generator $G$ of $\mathcal{G}$ as in  Theorem \ref{unexpo}. Take any $0< \eta <  \max\{\frac{\de}{4},\delta_0\}$ and let $N\in \N$ be given by the uniform $CW$-expansiveness with respect to $\eta$ giben again by Theorem \ref{unexpo}.

Fix some $g\in G$ and $D_g$. Since $M$ is locally connected and infinite, we can find $x\in M$ such that its connected component $C(x)$ is non-trivial. Therefore $CW$-expansiveness implies  that we can find some  $g\in \mathcal{G}$ such that $g(C(x))$ has some connected subset $C_0$ such that $diam(C_0)\geq \de$. In particular, this implies the existence of $2$ points in $C(x)$ which are $k$-$\de$-$G$separated, where $k=\#g$.

Since $M$ is locally connected, we can find two disjoint and connected subsets that $C_1^0$ and $C_1^1$ of $C_0$ with diameter grater than $\eta$. Now, uniform $CW$-expansiveness implies that there are $g_1^0,g_1^1$, satisfying $\#g_1^0,\#g_1^1\leq N$ and such that $diam(g_1^0(C_1^0))>\de$ and  $diam(g_1^1(C_1^1)>\de$.
This implies in the $C(x)$ the existence of $4$ points which are $(k+N)$-$\de$-$G$-generated.

If we take again on each set $C_1^0$ and $C_1^1$ two connected  and disjoint sets with diameter at least $\eta$, then we can find a set of $8$ points in $C(x)$ which are $(k+2N)$-$\de$-$G$-separated.

Inductively, the previous construction give the existence of $(k+nN)$-$\de$-$G$-separated sets $E_n$, with at least $2^{n+1}$ elements. Therefore,  the geometrical entropy of $\mathcal{G}$ with respect to $G$ satisfies the following:

$$h(\mathcal{G},G)=\lim_{\eps\to 0}\limsup_{n \to \infty}\frac{1}{n}\log S(n,\eps,G)\geq \lim_{n\to \infty}\frac{1}{k+nN}\log\#E_n=\frac{\log 2}{N}>0   $$

Since the positiveness of the geometric entropy does not depend on the  finite generator chosen, then the proof is complete.
\end{proof}

A direct application of Theorem \ref{EntPG} to the holonomy-pseudo of an expansive foliation gives Corollary \ref{EntFol}. Recall that the geometric entropy of a locally-free action is the geometric entropy of its orbit foliation. We are now able to prove Corollary \ref{ExpEnt}.

\begin{proof}[Proof of Corollary \ref{ExpEnt}] Suppose that $\varPhi$ is a locally-free expansive action of a connected Lie Group $G$ on $M$.
Since $\varPhi$ is locally free, we have that the orbits of $\varPhi$ generates a foliation $\mathcal{F}$ of $M$. By Theorem \ref{ExpFol}, $\mathcal{F}$ is expansive. Now Corollary \ref{EntFol} implies that $\mathcal{F}$ has positive geometric entropy and Corollary \ref{ExpEnt} is proved.

\section{Centralizers of Expansive Actions of Connected Lie Groups}

In this section we investigate the symmetries of actions of connected Lie Groups. We begin introducing the centralizer set of an action.  The study of the symmetries of dynamical systems is a classical problem which has  an algebraic flavor. It raised from group theory where it is interesting to know which elements of a given group $G$ commute with a fixed element  $g$ of $G$. To transpose this question to the dynamical systems scenario, we can think as follows. Let $M$ be a smooth manifold and let $Diff^r(M)$ denote the group of $C^r$-diffeomorphisms of $M$. If $G$ is a group, we can see a $C^r$-action of $G$ on $M$ as a  group-homomorphism $\rho:G \to Diff^r(M)$. 
Since $\rho(G)$ is a subgroup of $Diff^r(M)$, to study the symmetries of an $G$-action on $M$ is  equivalent to study which subgroups of $Diff^r(M)$ that are images of $G$-actions on $M$  commute with  $\rho(G)$. In this way, the problem of finding the symmetries $\rho(G)$ is an algebraic version of finding the actions which commutes $\varPhi$. 

Now let us precise the above comment. For a fixed $\varPhi\in \mathcal{A}^r(G.M)$, we say that $\varPsi\in \mathcal{A}^r(G,M)$ commutes with $\varPhi$ if $$\varPsi_g\circ\varPhi_h=\varPhi_h\circ\varPsi_g,$$
for every $g,h\in G$.
The $C^r$-centralizer of a $G$-action ,$\varPhi$, on $M$ is the set
$$\mathcal{C}^r(\varPhi)=\{\varPsi\in \mathcal{A}^r(G,M);\varPsi \textrm{ commutes with } \varPhi\}.$$

 \begin{definition}
An action $\varPhi$  of $G$ on $M$ has quasi-trivial $C^r$-centralizer if any $\varPsi\in \mathcal{C}^r(\varPhi) $ satisfies the following condition:
 There is some map $\xi:M\to End(G)$ constant along the orbits of $\varPhi$ and such that $\varPhi(\xi(g),\cdot)=\varPsi(g,\cdot)$ for any $g\in G$
 \end{definition}

For the case when $G=\R^d$, we mention the work \cite{BRV} of W. Bonomo, J. Rocha and P. Varandas, where it was proved the following:

\begin{theorem}[\cite{BRV}]\label{TBRV}
The $C^r$-centralizer of any expansive $\R^d$-action is quasi-trivial. 
\end{theorem}

 Our main goal here is to extended that result to more general Lie groups. Essentially, an action $\varPhi$ has quasi-trivial centralizer if any action commuting with $\varPsi$ has the same orbits as $\varPhi$, but the time is reparametrized by endomorphisms which only vary transversally to the orbits. We remark that previous definition naturally generalizes the respective definition of quasi-triviality for  $\R^d$-actions given in \cite{BRV}. 

Hereafter, we will proceed to obtain Theorem \ref{CC}.  Let us discuss the ideas behind this proof. In \cite{BRV}, the authors started with an expansive $\R^k$-action $\varPhi$ and given any other action $\varPsi$ commuting with $\varPhi$ it was possible to find a local group homomorphism which locally reparametrizes $\R^k$. 
The hard task here is to extend this local homomorphism to a global endomorphism of $\R^k$. 
This extension was strongly supported on the vector space structure of $\R^k$. 

Now if we are working with general Lie groups, we do not have an available vector space structure for $G$. But we have a natural vector space associated to $G$, namely the  Lie algebra $\mathfrak{G}$ of $G$. Recall that $\mathfrak{G}$ is isomorphic to $T_eM$. Here we are denoting $\exp$ for the exponential map at the identity element of $G$. To use the structure of vector space of $\mathfrak{G}$ we will see $\mathfrak{G}$ as an additive group and suppose that $\exp$ is a surjective group-homomorphism. Observe that a group  $G$ under the previous assumption, must be abelian. Clearly $\R^k$ satisfies previous assumption, but there are other examples of  such groups such as cylinders and more general  products of an abelian compact lie groups with some $\R^k$.

The idea behind our generalization is that starting with an expansive action of $G$ one can obtain a related expansive action of $\R^k$, if the group $G$ has an adequate structure. Suppose that $G$ is a  Lie group under the hypothesis of Theorem \ref{CC}. Recall that by the group isomorphism theorem we have that $\faktor{\mathfrak{G}}{Ker(\exp_e)}$ is isomorphic to $G$. Let us denote $\rho$ for this isomorphism and recall that $\rho$ is the factor map of $\exp$. Now, given an action $\varPhi:G\times M\to M$, we can use $\rho$ to induce an action $$\varPhi':\faktor{\mathfrak{G}}{Ker(\exp)}\times M \to M$$ as follows:

$$  \varPhi'(v,x)=\varPhi(\rho(v),x)$$

Next proposition is an elementary consequence of the definitions.

\begin{proposition}\label{expiso}
If $\varPhi$ is expansive, then $\varPhi'$ is expansive.
\end{proposition}

\begin{proof}
Suppose that $\varPhi$ is expansive. First, we fix $\eps>0$ small enough such that the exponential map is a local isometry from $B_{\eps}(0)$ to $G$. Observe that if we consider the metric induced by $\exp$ into $\faktor{\mathfrak{G}}{Ker(\exp)}$,  then  the same $\eps>0$ makes $\rho$ a local isometry from $\faktor{\mathfrak{G}}{Ker(\exp)}$ to $G$.   Now let $\de>0$ be given by the expansiveness of $\varPhi$. Suppose that there are $x,y\in M$ and a continuous  map $$\eta:\faktor{\mathfrak{G}}{Ker(\exp)}\to \faktor{\mathfrak{G}}{Ker(\exp)}$$ satisfying $\eta(e)=e$ such that $$d(\varPhi'_v(x),\varPhi'_{\eta(v)}(y))\leq \de$$ for any $v\in \faktor{\mathfrak{G}}{Ker(\exp)}$. 

Since $\rho$ is a group isomorphism, then $$\rho'=\rho\circ\eta\circ\rho^{-1}:G\to G$$ is a continuous map fixing $e$. Moreover, we have that $$d(\varPhi_{g}(x),\varPhi_{\rho(\eta(\rho^{-1}(g)))}(y))= d(\varPhi'_{\rho^{-1}(g)}(x),\varPhi'_{\eta(\rho^{-1}(g)}(y))\leq \de $$ for any $g\in G$. 

Thus, there is some $g_0\in B_{\eps}(e)$ such that $\varPhi_{g_0}(y)=x$. But this implies that $\varPhi'_{\rho^{-1}(g_0)}(x)=y$ and then $\varPhi'$ is expansive.
\end{proof}

Next suppose that $\varPhi$ and $\varPsi$ are two commuting $G$-actions on $M$. Then for any $v,u\in \mathfrak{G}$ we have the following:

$$\varPhi'_{v}\circ\varPsi'_{u}=\varPhi_{\rho(v)}\circ \varPsi_{\rho(u)}=\varPsi_{\rho(u)}\circ \varPhi_{\rho(v)}=\varPsi'_{u}\circ\varPhi'_{v}$$

But previous observations easily imply the following result:

\begin{proposition}\label{centiso}
For any $r\geq 0$ one has $\varPsi\in \mathcal{C}^r(\varPhi)$ if, and only if $\varPsi'\in \mathcal{C}^r(\varPhi')$
\end{proposition}

\end{proof}

Now we already have all the necessary elements to prove the main theorem of this section.

\begin{proof}[Proof of Theorem \ref{CC}]
Let $\varPhi$ be an expansive action and suppose that $\exp:\mathfrak{G}\to G$ is a  group-homomorphism. Fix $\varPhi'\in \mathcal{C}^r(\varPhi)$. Let $$\rho:\faktor{\mathfrak{G}}{Ker(\exp)}\to G$$ be the factor isomorphism of $\exp$. Let $\varPsi'$ be the action induced in $\faktor{\mathfrak{G}}{Ker(\exp)}$ by $\rho$.

Recall that $\faktor{\mathfrak{G}}{Ker(\exp)}$ is a finite dimensional real vector space, then it is isomorphic to some $\R^n$. Therefore, $\varPhi'$ and $\varPsi'$ can be seen as  actions of $\R^n$ on $M$. By Propositions \ref{expiso} and \ref{centiso}  $\varPhi'$ is expansive and $\varPsi'\in \mathcal{C}^r(\varPhi')$.  Now Theorem \ref{TBRV} implies that for any $x\in M$ there is a group-endomorphism $\eta_x$ of $\mathfrak{G}$ such that $\varPsi'_{\eta(v)}(x)=\varPhi'_{v}(x)$ for any $x\in M$ and $v\in \mathfrak{G}$, satisfying $\eta_x=\eta_{y}$ for any $y\in O_{\varPhi'}(x)$. 

Define a family of endomorphism of $G$ by $$\eta'_{x}=\rho\circ\eta_x\circ\rho^{-1}.$$ 
Now, this  implies that $$\varPsi_{g}(x)=\varPsi'_{\rho^{-1}(g)}(x)=\varPhi'_{\eta(\rho^{-1}(g))}(x)=\varPhi_{\rho(\eta_x(\rho^{-1}(g)))}(x)$$

For every $x\in M$ and $g\in G$. It is clear that $\eta'_x=\eta'_y$ for any $y\in O_{\varPhi}(x)$ and this concludes the proof.

\end{proof}

\section{ Expansive Actions of Finitely Generated Groups}

 Through this section $\varPhi$ will denote an action of a finitely generated group $G$. We begin by giving some examples of expansive actions.  An easy way to find expansiveness for an action of a finitely generated group is to prove that some of its generators is expansive. On the other hand, we can obtain expansive actions from non-expansive generators. 

 \begin{example}\label{DE}
We will consider an  action generated by two homeomorphisms $f$ and $g$.  Let $f$ be an irrational translation of the torus $\mathbb{T}^2$.  Let $T$ be a linear anosov map on $\mathbb{T}^2$ and let $p$ be its fixed point. Blow up $p$ in to a small disc $D$ and define $g$ on $T^2$ to be an extension of $T$ to $D$ as the identity map on $D$. It is clear that $f$ and $g$ are not expansive. Let $\Phi$ be the  action on $\T^2$ of the group generated by $f$ and $g$.

 We claim that $\phi$ is expansive. Notice that there exists $e>0$ such that any two distinct points in $\T^2\setminus D$ are $e$ apart at some time under the action of $g$. So we just need to consider the case when $x,y\in D$. To do that just notice that there exists $n>0$ such that at least one of $f^n(x)$ and $f^n(y)$ is outside of $D$. Now we can apply $g$ until $f^n(x)$ and $f^n(y)$ be $e$-apart. This proves that $e$ is an expansive constant for $\Phi$.
\end{example}
 
In \cite{Hur}  S. Hurder studied expansive actions induced by circle homeomorphisms. Next example is one of these actions and ilustrates that expansiveness can be obtainned in a way that none of its induced homeomorphisms is expansive. 

\begin{example}
Let us consider the homeomorphisms $f_1,f_2: S^1\to S^1$ such that $f_1$ is a irrational rotation and $f_2$ is a morse-smale homeomorphism with exactly two fixed points, a source $p_1$ and a sink $p_2$. Now let $\varPhi$ be the action on $S^1$ generated by $f_1$ and $f_2$. It is easy to see that $\varPhi$ is expansive. Indeed, fix $\de<d(p_1,p_2)$ and take $0<e<\frac{\de}{10}$. Notice $p_1$ and $p_2$ divide the circle in two distinct connected arcs and if two points $x,y$ are in distinct arcs, then they will be $e$-apart at some time by the iteration of $f_2$ or $f_2^{-1}$. If $x$ and $y$ are in the same arc, then we can apply the $f_1$ on until they belong to distinct arcs. Thus, we just need to apply $f_2$ or $f_2^{-1}$ until see  the desired separation. In this example, none of the homeomorphisms induced by $\varPhi$ can be expansive since they are defined on the circle.     
\end{example}

In this context, Corollary \ref{FG}  is obvious consequence of Theorem \ref{EntPG}, since any action of a finitely generated group, can be seen as the action of finitely generated pseudo-group.
  
Now we proceed to prove Theorem \ref{DC}. 
Let $\varPhi$ and $\varPsi$ be two $G$-actions on $M$. We denote $d_0$ for the $C^0$-distance on the space $\mathcal{A}^0(G,M)$, precisely, if $K$ is a finite generator of $G$, then 

$$d_0(\varPhi,\varPsi)=\max_{g_i\in K}\sup_{x\in M}\{d(\varPhi_{g_i}(x),\varPsi_{g_i}(x))\}  $$

Next definition is a generalization of the concept of discrete centralizer for homeomorphisms in \cite{W}.

\begin{definition}
We say that a continuous $G$-action $\varPhi$ has discrete $C^0$-centralizer if $\mathcal{C}^0(\varPhi)$ is a discrete subset of $\mathcal{A}^0(G,M)$ on the $C^0$-topology.
\end{definition}

\begin{proof}[Proof of Theorem \ref{DC}]
Let $\varPhi$ be an expansive $C^0$-action of $G$ on $M$ with expansiveness constant $\de>0$. Suppose that there are $\varPsi,\varPsi'\in \mathcal{C}^0(\varPhi)$ satisfying $d_0(\varPsi,\varPsi')\leq \de$.  If $\varPsi\neq \varPsi'$, we can find $x\in M$ and $g\in G$ such that $\varPsi_g(x)\neq \varPsi'_g(x)$. Now, for any $h\in G$ we have that $$ d(\varPhi_h(\varPsi_g(x),\varPhi_h(\varPsi'_g(x)))=d(\varPsi_g(\varPhi_h(x)),\varPsi'_g(\varPhi_h(x)))\leq \de $$ 

But this contradicts the expansiveness of $\varPhi$, thus $\varPsi=\varPsi'$ and $\mathcal{C}^0(\varPhi)$ is discrete. 
\end{proof}

\vspace{0.1in} 
\textbf{Acknowledgements:} \textit{The authors would like to thank professor Pablo Daniel Carrasco for his great help in the development of this work. His comments  were essential to improve the content of this work.
The authors also would like to thank the referees. Their valuable  comments and suggestions were very useful and helped the authors to improve the exposition of this work. }

\end{document}